\documentclass{amsart}

\usepackage{amssymb}

\usepackage[all]{xy}

\newtheorem{thm}{Theorem}

\newtheorem{cor}[thm]{Corollary}

\theoremstyle{definition}

\newtheorem{rem}[thm]{Remark}

\newtheorem{defn-thm}[thm]{Definition--Theorem}  
\newtheorem{defn-lem}[thm]{Definition--Lemma}  

\theoremstyle{remark}

\renewcommand{\o}[0]{{\mathcal O}} 

\newcommand{\p}[0]{{\mathbb P}}

\newcommand{\sF}{\mathcal{F}}
\newcommand{\sG}{\mathcal{G}}

\def\loccoh#1.#2.#3.#4.{H^{#1}_{#2}(#3,#4)}

\DeclareMathAlphabet{\mathchanc}{OT1}{pzc}%
                                {m}{it}

\begin{document}
\bibliographystyle{amsalpha}


\title[Ulrich bundles on blowups]{Ulrich bundles on blowups}

\author{Yeongrak Kim}
\address{Simion Stoilow Institute of Mathematics of the Romanian Academy, P.O. Box 1-764, 014700 Bucharest, Romania}
\email{Yeongrak.Kim@gmail.com}
\thanks{The author has been partly supported by the CNCS-UEFISCDI Grant PN-II-PCE-2011-3-0288 and BITDEFENDER postdoctoral fellowship.}

\begin{abstract}
We construct an Ulrich bundle on the blowup at a point when the original variety is embedded by a sufficiently positive linear system and carries an Ulrich bundle. In particular, we describe the relation between special Ulrich bundles on the blown-up surfaces and the original surfaces.
\end{abstract}

\maketitle

Let $X \subset \p^N$ be a smooth projective variety of dimension $n$, embedded by a complete linear system $| \o_X(H) |$ for some very ample divisor $H$. An \emph{Ulrich bundle} on $X$ \cite{ESW} is a vector bundle $\sF$ on $X$ whose twists satisfy a set of vanishing conditions on cohomology
\[
H^i (X, \sF(-jH)) = 0 \textrm{ for all } i \textrm{ and } 1\le j \le n. 
\]
Ulrich bundles appeared in commutative algebra in relation with maximally generated maximal Cohen-Macaulay modules \cite{U84}. In algebraic geometry, the notion of Ulrich bundles suprisingly appeared thanks to recent works by Beauville and Eisenbud-Schreyer. The importance is motivated by the relations between the Cayley-Chow forms \cite{B00,ESW} and with the cohomology tables \cite{ES11}.

Eisenbud and Schreyer made a conjecture that every projective variety admits an Ulrich bundle \cite{ESW}, which is wildly open even for smooth surfaces. The answer is known for a few cases including: curves \cite{ESW}, complete intersections \cite{HUB}, Grassmannians \cite{CM}, del Pezzo surfaces \cite{ESW} and more rational surfaces with an anticanonical pencil \cite{Kim16}, general K3 surfaces \cite{AFO}, abelian surfaces \cite{B15}, Fano polarized Enriques surfaces \cite{BN}, and surfaces with $q=p_g=0$ embedded by a sufficiently large linear system \cite{B16}. 

In classical algebraic geometry, there are 2 fundamental operations, namely, the hyperplane cut and the linear projection. It is well known that the restriction of an Ulrich bundle to a general hyperplane section is also an Ulrich bundle (cf. \cite{CH}). It is also straightforward that the vanishing conditions do not affect on taking a linear projection from a point $P \in \p^N \setminus X$. Hence, the only interesting case occurs from the ``projection'' from a point inside of $X$ which can be realized as the blowup at a point.

We briefly review the relation between inner projections and blowups. Let $P \in X$ be a point. The linear projection from $P$ gives a rational map $\pi_P : X \dashrightarrow \p^{N-1}$ defined on $X \setminus \{ P \}$. We can eliminate the point of indeterminacy by taking the blow-up $\sigma : \tilde{X} \to X$ at $P$. The complete linear system $| \sigma^{*} \o_{X}(H) \otimes \o_{\tilde{X}}(-E) |$ induces a morphism from $\tilde{X}$ to $\p^{N-1}$ whose image is the closure of ${\pi_P(X \setminus \{ P \})}$, where $E = \sigma^{-1}(P)$ is the exceptional divisor. 

In this short note, we construct an Ulrich bundle on the blowup at a point from an Ulrich bundle on the original variety. 

\begin{thm}\label{Theorem:Main}
Assume furthermore that the divisor $\tilde{H} := \sigma^* H - E$ is very ample. Suppose we have an Ulrich bundle $\sF$ on $X$ with respect to the polarization $\o_X(H)$. Then the vector bundle 
\[
\tilde{\sF} := \sigma^* \sF \otimes \o_{\tilde{X}} (-E)
\]
is an Ulrich vector bundle on $\tilde{X}$ with respect to $\o_{\tilde{X}} (\tilde{H})$.
\end{thm}

\begin{proof}
We have to show that $\tilde{\sF} (-j \tilde{H}) = \sigma^{*}( \sF(-jH)) \otimes \o_{\tilde{X}} ( (j-1)E)$ has no cohomology for every $1 \le j \le n$. Note that the push-forward $\sigma_{*} \o_{\tilde{X}} (jE) = \o_X$ for every $j \ge 0$. We first claim that the higher direct image $R^i \sigma_{*} \o_{\tilde{X}} ((j-1)E) = 0$ for every $i>0$ and $1 \le j \le n$. It is enough to show that the stalk vanishes at every point $Q \in X$, which can be computed from the inverse limit
\[
\left( R^i \sigma_{*} \o_{\tilde{X}} ((j-1)E) \right)^{\wedge}_{Q} = 
\left\{\begin{array}{lr}
0 & Q \neq P, \\
\varprojlim H^i (mE, \o_{mE} ((j-1)E)) & Q = P.  \\
\end{array}\right.
\]
By the short exact sequence
\[
0 \to \o_E (-(m-1)E) \simeq \o_{\p^{n-1}} (m-1) \to \o_{mE} \to \o_{(m-1)E} \to 0,
\]
we have
\begin{eqnarray*}
H^i (mE, \o_{mE}((j-1)E)) & \simeq & H^i ((m-1)E, \o_{(m-1)E}((j-1)E)) \\
& \vdots & \\
& \simeq & H^i (E, \o_E ((j-1)E)) \\
& = & H^i (\p^{n-1} , \o_{\p^{n-1}} (1-j)) = 0
\end{eqnarray*}
for any $i>0, m \ge 1$ and $1 \le j \le n$. 
Applying the projection formula, we have
\begin{eqnarray*}
R^i \sigma_{*} ( \tilde{\sF} (-j \tilde{H})) & = & \sF(-jH) \otimes R^i \sigma_{*} \o_{\tilde{X}} ((j-1)E) \\
& = & 0
\end{eqnarray*}
for every $i>0$ and $1 \le j \le n$. Hence, Leray spectral sequence implies that the cohomology group
\begin{eqnarray*}
H^i (\tilde{X}, \tilde{\sF} (-j \tilde{H})) & \simeq & H^i (X, \sigma_{*} (\tilde{\sF}(-j \tilde{H}))) \\
& \simeq & H^i (X, \sF(-jH) \otimes \sigma_{*} \o_{\tilde{X}} ((j-1)E)) \\
& = & H^i (X, \sF(-jH)) \\
& = & 0
\end{eqnarray*}
vanishes for every $i$ and $1 \le j \le n$, since $\sF$ is Ulrich on $(X,H)$. Therefore, we conclude that $\tilde{\sF}$ is an Ulrich vector bundle on $( \tilde{X}, \tilde{H})$. 
\end{proof}

Particularly interesting case happens when $X$ is a smooth surface. The above construction provides an Ulrich bundle on blown-up surfaces at a few points, by taking consecutive inner projections. Moreover, the procedure also provides a direct application on ``special Ulrich bundles''. Eisenbud and Schreyer introduced the notion of \emph{special Ulrich bundles} on a surface $X$ \cite{ESW}, which are Ulrich bundles $\sF$ of rank 2 such that $\det \sF = \o_X(K_X + 3H)$, where $K_X$ denotes the canonical divisor of $X$. The existence of special Ulrich bundles yields a very nice presentation of the Cayley-Chow form of $X$, indeed, $X$ admits a Pfaffian B{\'e}zout form in Pl{\"u}cker coordinates \cite{ESW}. As an immediate consequence, the procedure with a special Ulrich bundle gives rise to a special Ulrich bundle on the upstairs:

\begin{cor}
Let $(X,H)$ be a smooth polarized surface satisfies the assumptions in \emph{Theorem \ref{Theorem:Main}}. If $\sF$ is a special Ulrich bundle on $X$, then $\tilde{\sF}$ is also a special Ulrich bundle on $\tilde{X}$.
\end{cor}
\begin{proof}
It comes from a direct computation 
\[ 
\det \tilde{\sF} = \sigma^{*}( \o_X(K_X + 3H)) \otimes \o_{\tilde{X}}(-2E) = \o_{\tilde{X}}(K_{\tilde{X}} + 3 \tilde{H}).
\]
\end{proof}

It is also possible to construct a special Ulrich bundle in the converse direction, which reveals the connection between special Ulrich bundles on the upstairs and downstairs:

\begin{thm}
Let $(X,H)$ be a smooth polarized surface satisfies the assumptions in \emph{Thoerem \ref{Theorem:Main}} as above. Let $\tilde{\sF}$ be a special Ulrich bundle on $(\tilde{X}, \tilde{H})$. Then $\sigma_{*} (\tilde{\sF}(E))$ is a special Ulrich bundle on $(X,H)$.
\end{thm}

\begin{proof}
We first claim that $\sigma_{*} (\tilde{\sF}(E))$ is a vector bundle on $X$. Since $c_1(\tilde{\sF}(E)) = K_{\tilde{X}} + 3 \tilde{H} + 2E = \sigma^{*} (K_X + 3H)$, we have $\deg \tilde{\sF}(E)|_E =0$. By Grothendieck's theorem, we have $\tilde{\sF}(E)|_E \simeq \o_{\p^1}(a) \oplus \o_{\p^1}(-a)$ for some $a \ge 0$. Note that $\tilde{\sF}$ is globally generated since it is 0-regular with respect to $\tilde{H}$. Hence the restriction $\tilde{\sF}|_E \simeq \o_{\p^1}(a+1) \oplus \o_{\p^1}(-a+1)$ is also globally generated, so either $a=0$ or $a=1$ holds. For the both cases, we obtain $h^0(E, \tilde{\sF}(E)|_E) = 2$. Therefore $h^0(\sigma^{-1}(Q), \tilde{\sF}(E)|_{\sigma^{-1}(Q)}) = 2$ holds for every $Q \in X$, which implies that $\sigma_{*} (\tilde{\sF}(E))$ is locally free of rank 2 by Grauert's theorem. Also note that a similar computation induces that $R^1 \sigma_{*}(\tilde{\sF}(E)) = 0$ since $\tilde{\sF}(E)|_E$ has vanishing $H^1$. 

To prove $\sigma_{*} (\tilde{\sF}(E))$ is a special Ulrich bundle, it is enough to show that 
$\det \sigma_{*} (\tilde{\sF}(E))$ $ \simeq \o_X (K_X + 3H)$ and partial vanishing conditions $H^{\bullet}(X, \sigma_{*} (\tilde{\sF}(E))\otimes \o_X(-H)) = 0$. Indeed, for those vector bundles, we have
\begin{eqnarray*}
H^i (X, \sigma_{*} (\tilde{\sF}(E))\otimes \o_X(-2H)) & = & H^{2-i} (X, \sigma_{*} (\tilde{\sF}(E))^{*} \otimes \o_X(K_X+2H))^{*} \\
& = & H^{2-i}(X, \sigma_{*}(\tilde{\sF}(E)) \otimes \o_X(-H))^{*} \\
& = & 0
\end{eqnarray*}
by Serre duality. 

Note that the determinant of a coherent sheaf $\sG$ is the alternating product $\bigotimes \det (\mathcal{E}_i)^{(-1)^i}$ where $\mathcal{E}_i$ define
\[
0 \to \mathcal{E}_r \to \cdots \to \mathcal{E}_1 \to \mathcal{E}_0 \to \sG \to 0
\]
a finite locally free resolution of $\sG$. It is well-known that such a locally free resolution always exists on a smooth variety, and the determinant is equal to the structure sheaf when $\sG$ is supported on a subset of codimension at least 2 (cf. \cite{HN}).
Let $V \subset H^0(\tilde{X}, \tilde{\sF})$ be a general subspace of dimension 3. Since $\tilde{\sF}$ is globally generated, the evaluation map $ev : V \otimes \o_{\tilde{X}} \to \tilde{\sF}$ is surjective possibly except for finitely many points. Hence we have an exact sequence
\[
0 \to \o_{\tilde{X}} (-K_{\tilde{X}} - 3 \tilde{H}) \simeq \sigma^{*} \o_X (-K_X - 3H) \otimes \o_{\tilde{X}}(2E) \to V \otimes \o_{\tilde {X}} \stackrel{ev} \longrightarrow \tilde{\sF} \to \mathcal{R}_{Z} \to 0.
\]
Here, $\mathcal{R}_Z = (\ker ev)$ is supported on a finite set of points $Z \subset \tilde{X}$. Since $\mathcal{R}_Z$ and $R^1 \sigma_{*}$ terms have supports of codimension at least 2, they don't affect on the determinant computation. Twisting by $\o_{\tilde{X}}(E)$ and taking push-forward, we conclude that
\begin{eqnarray*}
\det \sigma_{*} (\tilde{\sF}(E)) & = & \det ( V \otimes \sigma_{*} \o_{\tilde{X}}(E)) \otimes (\det \sigma_{*} (\sigma^{*} \o_X(-K_X - 3H) \otimes \o_{\tilde{X}}(3E)))^{*} \\
&= & (\wedge^3 V \otimes \o_X) \otimes \o_X(-K_X - 3H)^{*} \\
& = & \o_X(K_X + 3H)
\end{eqnarray*}
as desired.
 
Apply the projection formula and Leray spectral sequence, we have
\begin{eqnarray*}
H^i(X, \sigma_{*} (\tilde{\sF}(E))\otimes \o_X(-H)) & \simeq & H^i (\tilde{X}, \tilde{\sF} (E) \otimes \sigma^{*} \o_X (-H)) \\
& = & H^i (\tilde{X}, \tilde{\sF}(-\tilde{H})) \\
& = & 0
\end{eqnarray*}
which completes the proof.
\end{proof}

\begin{rem}

When $X$ is a smooth regular surface, Noma found an equivalence condition for the very ampleness of the analogus line bundle in terms of the position of points for the blowup center \cite{N06}. 

\end{rem}

\bibliography{refs-main/refs}
\def\cprime{$'$} \def\cprime{$'$} \def\cprime{$'$} \def\cprime{$'$}
  \def\cprime{$'$} \def\cprime{$'$} \def\dbar{\leavevmode\hbox to
  0pt{\hskip.2ex \accent"16\hss}d} \def\cprime{$'$} \def\cprime{$'$}
  \def\polhk#1{\setbox0=\hbox{#1}{\ooalign{\hidewidth
  \lower1.5ex\hbox{`}\hidewidth\crcr\unhbox0}}} \def\cprime{$'$}
  \def\cprime{$'$} \def\cprime{$'$} \def\cprime{$'$}
  \def\polhk#1{\setbox0=\hbox{#1}{\ooalign{\hidewidth
  \lower1.5ex\hbox{`}\hidewidth\crcr\unhbox0}}} \def\cdprime{$''$}
  \def\cprime{$'$} \def\cprime{$'$} \def\cprime{$'$} \def\cprime{$'$}
\providecommand{\bysame}{\leavevmode\hbox to3em{\hrulefill}\thinspace}
\providecommand{\MR}{\relax\ifhmode\unskip\space\fi MR }
\providecommand{\MRhref}[2]{%
  \href{http://www.ams.org/mathscinet-getitem?mr=#1}{#2}
}
\providecommand{\href}[2]{#2}

\vskip1cm

\end{document}